\documentclass[a4paper,12pt]{article}

\usepackage{amsmath,amssymb,amsthm}
\usepackage{amssymb,latexsym, bbm,comment}

\renewcommand{\theequation}                            
       {\mbox{\arabic{section}.\arabic{equation}}}

\newcommand{\origsetminus}{} \let\origsetminus=\setminus           
\renewcommand{\setminus}{\!\origsetminus\!}

{\theoremstyle{plain}
\newtheorem{definition}{Definition}[section]
\newtheorem{lemma}[definition]{Lemma}
\newtheorem{theorem}[definition]{Theorem}
\newtheorem{corollary}[definition]{Corollary}

}
{\theoremstyle{definition}
\newtheorem{example}[definition]{Example}
\newtheorem{remark}[definition]{Remark}

}

\renewcommand{\mathbb}{\mathbbm}                     
\renewcommand{\epsilon}{\varepsilon}                 
\renewcommand{\phi}{\varphi}
\renewcommand{\theta}{\vartheta}
\renewcommand{\le}{\leqslant}
\renewcommand{\ge}{\geqslant}

\newcommand{\origfoo}{} \let\origfoo=\sqrt           
\renewcommand{\sqrt}[1]{\origfoo{#1}\;}

\renewcommand{\O}{{\mathcal O}}                      
\newcommand{\abs}[1]{\left\lvert #1 \right\rvert}    
\newcommand{\norm}[1]{\left\lVert #1 \right\rVert}   
\DeclareMathOperator{\R}{{\mathbb R}}                
\DeclareMathOperator{\C}{{\mathbb C}}                
\DeclareMathOperator{\N}{{\mathbb N}}                
\DeclareMathOperator{\Id}{ Id}                        

\newcommand{\A}{{\mathcal A}}
\DeclareMathOperator{\Borel}{{\mathcal B}}
\newcommand{\scapro}[2]{\langle #1,#2\rangle}       
\DeclareMathOperator{\1}{\mathbbm 1}
\renewcommand{\S}{{\mathcal S}}

\DeclareMathOperator{\Z}{{\mathcal Z}} \DeclareMathOperator{\Cc}{{\mathcal C}}

\newcounter{zahl}

\DeclareMathOperator{\fa}{\qquad\text{for all }}

\title{Infinitely divisible cylindrical measures \\on Banach spaces}

\author{ Markus Riedle\\
Department of Mathematics\\
King's College London\\
London WC2R 2LS\\
United Kingdom \\
\texttt{markus.riedle@kcl.ac.uk}}

\begin{document}
\maketitle

\begin{abstract}
In this work infinitely divisible cylindrical probability measures on arbitrary Banach
spaces are introduced. The class of infinitely divisible cylindrical probability measures
is described in terms of their characteristics, a characterisation which is not known in
general for infinitely divisible Radon measures on Banach spaces. Further properties of
infinitely divisible cylindrical measures such as continuity are derived. Moreover, the
result on the classification enables us to conclude new results on genuine L{\'e}vy
measures on Banach spaces.
\end{abstract}

\section{Introduction}

Probability theory in Banach spaces has been extensively studied since 1960 and several
monographs are dedicated to this field of mathematics, e.g. de Araujo and Gin{\'e}
\cite{AraujoGine}, Ledoux and Talagrand \cite{LedTal} and Vakhania et al \cite{Vaketal}.
This field of mathematics is closely related to the theory of Banach space geometry and
it has applications not only in probability theory but also in operator theory, harmonic
analysis and C$^\ast$-algebras.

Cylindrical stochastic processes in Banach spaces appear naturally as the driving noise
in stochastic differential equations in infinite dimensions, such as stochastic partial
differential equations and interest rate models. Up to now, cylindrical Wiener processes
are the standard examples of the driving noise, which restricts the noise to a Gaussian
perturbation with continuous paths.  A natural non-Gaussian and discontinuous
generalisation is introduced by {\em cylindrical L{\'e}vy processes}. The notion
cylindrical L{\'e}vy process appears the first time in Peszat and Zabczyk \cite{PesZab}
and it is followed by the works Brze\'zniak, Goldys et al \cite{Zdzisetal}, Brze\'zniak
and Zabzcyk \cite{BrzZab} and Priola and Zabczyk \cite{PriolaJerzy2}. The first
systematic introduction of cylindrical L{\'e}vy processes appears in our work Applebaum
and Riedle \cite{DaveMarkus}.

The introduction of cylindrical L{\'e}vy processes in \cite{DaveMarkus}  are based on the theory of cylindrical or generalised processes and cylindrical measures, see for example Schwartz \cite{Schwartz} or Vakhaniya et al \cite{Vaketal}. This approach
in \cite{DaveMarkus} is inspired by the analogue definition for cylindrical Wiener processes, see Kallianpur and  Xiong \cite{Kallianpur}, Metivier and Pellaumail \cite{MetivierPell} or Riedle \cite{riedlewiener}. In the same way as cylindrical Wiener processes are related to the class of Gaussian cylindrical measures, the introduction of cylindrical L{\'e}vy processes in \cite{DaveMarkus} leads to the new class of infinitely divisible cylindrical measures which have not been considered so far. Since the article \cite{DaveMarkus}  is focused on cylindrical L{\'e}vy processes and their stochastic integral, no further properties of infinitely divisible cylindrical measures are derived. In this work we give a rigorous introduction of
infinitely divisible cylindrical measures in Banach spaces and derive some fundamental properties of them.
Some of the results also give a new insight on genuine infinitely divisible Radon
measures on Banach spaces.

The main result is the characterisation of the class of infinitely divisible cylindrical measures in a Banach space in terms
of a triplet $(p,q,\nu)$ where $p,q$ are some functions and $\nu$ is a cylindrical measure. This result is surprising
since for infinitely divisible Radon measures in Banach spaces such a classification is not known in general, see de Araujo and Gin\'e \cite{AraujoGine}. Furthermore, since in analogy to the characteristics of L{\'e}vy processes the triplet describes the deterministic drift, covariance structure of the Gaussian part and jump distribution it provides the construction of an infinitely divisible cylindrical random variable for given data specifying these properties. Moreover, this main result enables us to derive the following two important conclusions.

The first one concerns the following problem: even in the finite dimensional
case,  a probability measure on $\R^2$ which satisfies that all image measures under
linear projections to $\R$ are infinitely divisible might be not infinitely divisible,
see Gin{\'e} and Hahn \cite{GineHahn} and Marcus \cite{Marcus}. However, a question left
open is if a probability measure on an infinite dimensional space is infinitely divisible
under the condition that all linear projections to $\R^n$ for all finite dimensions
$n\in\N$ are infinitely divisible? By the characterisation of the set of infinitely
divisible cylindrical measures mentioned above we are able to answer this question
affirmative.

The second conclusion of our main result concerns the characterisation of L{\'e}vy measures in Banach spaces.
In a Hilbert space $H$ it is well known that a $\sigma$-finite measure $\nu$ is the L{\'e}vy measure of an infinitely divisible
Radon measure if and only if
\begin{align}\label{eq.intro}
  \int_H \Big(\norm{u}^2\wedge 1\Big)\,\nu(du)<\infty,
\end{align}
see for example Parthasarathy \cite{Parth67}. Although this integrability condition can
be used to classify the type and cotype of Banach spaces, see de Araujo and Gin{\'e}
\cite{AraujoGine2}, in general Banach spaces such an explicit description of infinitely
divisible measures in terms of the L{\'e}vy measure $\nu$ is not known. Even worse, the
condition \eqref{eq.intro} might be neither sufficient nor necessary for a
$\sigma$-finite measure $\nu$ on an arbitrary Banach space $U$ to guarantee that there
exists an infinitely divisible measure with characteristics $(0,0,\nu)$, see for example
$U=C[0,1]$ in  Araujo \cite{Araujo75}. However, we show in the last part of this work
that a $\sigma$-finite measure $\nu$ satisfying the weaker  condition
\begin{align}
  \int_U \Big(\abs{\scapro{u}{a}}^2\wedge 1\Big)\,\nu(du)<\infty \fa a\in
  U^\ast,
\end{align}
always generates an infinitely divisible cylindrical measure $\mu$. This result
reduces the question whether a $\sigma$-finite measure $\nu$ generates  an infinitely
divisible Radon measure to the question whether the infinitely divisible cylindrical
measure $\mu$ extends to a Radon measure.

\section{Preliminaries}

For a measure space $(S,\S,\mu)$ we denote by $L^p_\mu(S,\S)$, $p\ge 0$ the space of
equivalence classes of measurable functions $f:S\to\R$ which satisfy $\int
\abs{f(s)}^p\,\mu(ds)<\infty$.

Let $U$ be a Banach space with dual $U^\ast$. The dual pairing is denoted by
$\scapro{u}{a}$ for $u\in U$ and $a\in U^\ast$. The Borel $\sigma$-algebra in $U$ is
denoted by  $\Borel(U)$ and the closed unit ball at the origin by $B_U:=\{u\in U:\,
\norm{u}\le 1\}$.

For every $a_1,\dots, a_n\in U^{\ast}$ and $n\in\N$ we define a linear map
\begin{align*}
  \pi_{a_1,\dots, a_n}:U\to \R^n,\qquad
   \pi_{a_1,\dots, a_n}(u)=(\scapro{u}{a_1},\dots,\scapro{u}{a_n}).
\end{align*}
Let $\Gamma$ be a subset of $U^\ast$. Sets of the form
 \begin{align*}
Z(a_1,\dots ,a_n;B):&= \{u\in U:\, (\scapro{u}{a_1},\dots,
 \scapro{u}{a_n})\in B\}\\
 &= \pi^{-1}_{a_1,\dots, a_n}(B),
\end{align*}
where $a_1,\dots, a_n\in \Gamma$ and $B\in \Borel(\R^n)$ are called {\em
cylindrical sets}. The set of all cylindrical sets is denoted by $\Z(U,\Gamma)$
and it is an algebra. The generated $\sigma$-algebra is denoted by
$\Cc(U,\Gamma)$ and it is called the {\em cylindrical $\sigma$-algebra with
respect to $(U,\Gamma)$}. If $\Gamma=U^\ast$ we write $\Z(U):=\Z(U,\Gamma)$ and
$\Cc(U):=\Cc(U,\Gamma)$.

A function $\mu:\Z(U)\to [0,\infty]$ is called a {\em cylindrical measure on
$\Z(U)$}, if for each finite subset $\Gamma\subseteq U^\ast$ the restriction of
$\mu$ to the $\sigma$-algebra $\Cc(U,\Gamma)$ is a measure. A cylindrical
measure is called finite if $\mu(U)<\infty$ and a cylindrical probability
measure if $\mu(U)=1$.

For every function $f:U\to\C$ which is measurable with respect to
$\Cc(U,\Gamma)$ for a finite subset $\Gamma\subseteq U^\ast$ the integral $\int
f(u)\,\mu(du)$ is well defined as a complex valued Lebesgue integral if it
exists. In particular, the characteristic function $\phi_\mu:U^\ast\to\C$ of a
finite cylindrical measure $\mu$ is defined by
\begin{align*}
 \phi_{\mu}(a):=\int_U e^{i\scapro{u}{a}}\,\mu(du)\qquad\text{for all }a\in
 U^\ast.
\end{align*}
In contrary to classical probability measures on $\Borel(U)$ there exists an analogue of Bochner's theorem for
cylindrical probability measures, see \cite[Prop.VI.3.2]{Vaketal}:  a function
$\phi:U^\ast\to \C$ with $\phi(0)=1$ is the characteristic function of a cylindrical
probability measure if and only if it is positive-definite and continuous on every
finite-dimensional subspace.

For every $a_1,\dots, a_n\in U^{\ast}$ we obtain an image measure
$\mu\circ\pi_{a_1,\dots, a_n}^{-1}$ on $\Borel(\R^n)$. Its characteristic
function $ \phi_{\mu\circ \pi_{a_1,\dots, a_n}^{-1}}$ is determined by that of
$\mu$:
\begin{align}\label{eq.charimcyl}
 \phi_{\mu\circ\pi_{a_1,\dots, a_n}^{-1}}(t)=
  \phi_{\mu}(t_1a_1+\cdots + t_n a_n)
\end{align}
for all $t=(t_1,\dots, t_n)\in \R^n$.

If $\mu_{1}$ and $\mu_{2}$ are cylindrical probability measures on $\Z(U)$ their
convolution is the cylindrical probability measure defined by
$$ (\mu_{1} * \mu_{2})(Z) = \int_{U} \mu_{1}(Z-u)\,\mu_{2}(du),$$ for each $Z \in \Z(U)$. Indeed if $Z =
\pi_{a_1,\dots, a_n}^{-1}(B)$ for some $a_1,\dots, a_n \in U^{\ast}, B \in
{\cal B}(\R^n)$, then it is easily verified that
\begin{equation*}
  (\mu_{1} * \mu_{2})(Z)
=  (\mu_{1} \circ \pi_{a_{a_1,\dots, a_n}}^{-1}) * (\mu_{2} \circ
\pi_{a_{a_1,\dots, a_n}}^{-1})(B).
\end{equation*}
A standard calculation yields $\phi_{\mu_{1} * \mu_{2}} = \phi_{\mu_{1}}\phi_{\mu_{2}}$. For more
information about convolution of cylindrical probability measures, see \cite{Ros}. The
$k$-times convolution of a cylindrical probability measure $\mu$ with itself is denoted
by $\mu^{\ast k}$.

\section{Infinitely divisible cylindrical measures}

For later reference, we begin with the well understood class of infinitely divisible
measures on $\R$. A probability measure $\zeta$ on $\Borel(\R)$ is called {\em infinitely
divisible} if for every $k\in\N$ there exists a probability measure $\zeta_k$ such that
$\zeta=(\zeta_k)^{\ast k}$. It is well known that infinitely divisible probability
measures on $\Borel(\R)$ are characterised by their characteristic function. The
characteristic function is unique but its specific representation depends on the chosen
truncation function.
\begin{definition}\label{de.truncation}
  A truncation function is any measurable function $h:\R\to\R$ which is bounded and satisfies $h=\Id$ in
  a neighborhood $D(h)$ of 0.
\end{definition}
 Given a truncation function $h$ a probability measure $\zeta$ on $\Borel(\R)$ is
infinitely divisible if and only if its characteristic function is of the form
\begin{align}\label{eq.charmuone}
  \phi_\zeta:\R\to\C,
  \qquad\phi_\zeta(t)=\exp\left(im t  -\tfrac{1}{2}r^2 t^2
     +\int_{\R}\widetilde{\psi}_h(s,t)\,\eta(ds)\right)
  \end{align}
for some constants $m\in \R$, $r\ge 0$ and a L{\'e}vy measure $\eta$, which is a
$\sigma$-finite measure $\eta$ on $\Borel(\R)$ with $\eta(\{0\})=0$ and
\begin{align*}
  \int_{\R} \left(\abs{s}^2\wedge 1\right)\,\eta(ds)<\infty.
\end{align*}
The function $\widetilde{\psi}_h$ is defined by
   \begin{align*}
      \widetilde{\psi}_h:\R\times \R\to\C,\qquad
      \widetilde{\psi}_h(s,t):=e^{ist}-1-ith(s).
   \end{align*}
In this situation we call the triplet $(m,r,\eta)_h$ the {\em characteristics of
$\zeta$}. If $h^\prime$ is another truncation function then $(m^\prime,
r,\eta)_{h^\prime}$ is the characteristics of $\zeta$ with respect to $h^\prime$, where
\begin{align*}
  m^\prime:=m +\int_{\R}(h^\prime(s)-h(s))\,\eta(ds).
\end{align*}
The integral exists because $h^\prime(s)-h(s)=0$ for $s\in D(h^\prime)\cap D(h)$ and $h$
and $h^\prime$ are both bounded. From Bochner's theorem and the Schoenberg's
correspondence (see \cite[Ch. IV.1.4]{Vaketal}) it follows that the function
\begin{align*}
  t\mapsto - \int_{\R}\tilde{\psi}_h(s,t)\,\eta(s)
\end{align*}
is negative-definite for all L{\'e}vy measures $\eta$. By choosing $\eta=\delta_{s_0}$, where $\delta_{s_0}$ denotes
the Dirac measure in $s_0$ for a constant $s_0\in\R$, we conclude that
\begin{align}\label{eq.psinegative}
  t\mapsto -\tilde{\psi}_h(s_0,t)\quad
  \text{is negative-definite for all }s_0\in\R.
\end{align}

Now, we move to the general situation of an arbitrary Banach space $U$. A Radon
probability measure $\mu$ on $\Borel(U)$ is called {\em infinitely divisible} if for each
$k\in\N$ there exists a Radon probability measure $\mu_k$ such that $\mu=(\mu_k)^{\ast
k}$. We generalise this definition to cylindrical measures:
\begin{definition}\label{de.infdivcyl}
  A cylindrical probability  measure $\mu$ on $\Z(U)$ is called {\em infinitely divisible}
  if there exists for each $k\in\N$ a cylindrical probability  measure $\mu_k$ such that
  $\mu=(\mu_k)^{\ast k}$.
\end{definition}
Bochner's theorem for cylindrical probability measures,  see for example \cite[Prop.VI.3.2]{Vaketal},
implies that a cylindrical probability  measure $\mu$ on $\Z(U)$ is infinitely divisible
if and only if for every $k\in\N$ there exists a characteristic function $\phi_{\mu_k}$
of a cylindrical probability measure $\mu_k$ such that
\begin{align*}
  \phi_{\mu}(a)=\left(\phi_{\mu_k}(a)\right)^k \fa a\in U^\ast.
\end{align*}
One might conjecture that a cylindrical probabability measure $\mu$ is infinitely
divisible if every image measure $\mu\circ a^{-1}$ is infinitely divisible for all $a\in
U^\ast$. But this is wrong already in the case $U=\R^2$ as shown by Gin{\'e} and Hahn
\cite{GineHahn} and Marcus \cite{Marcus}. They constructed a probability measure $\mu$ on
$\Borel(\R^2)$ such that all projections $\mu\circ a^{-1}$ are infinitely divisible for
all linear functions $a:\R^2\to\R$ but $\mu$ is not infinitely divisible. However, in
infinite dimensions one can require that all finite dimensional projections are
infinitely divisible.
\begin{definition}\label{de.weaklyinfdivcyl}
   A cylindrical probability  measure $\mu$ on $\Z(U)$ is called {\em weakly infinitely divisible}
 if and only if
\begin{align*}
  \mu\circ \pi_{a_1,\dots, a_n}^{-1}
  \text{ is infinitely divisible for all }
   a_1,\dots, a_n\in U^\ast \text{ and }n\in\N.
\end{align*}
\end{definition}
A cylindrical probability measure $\mu$ is weakly infinitely divisible if and only if for
each $k\in\N$ and all $a_1,\dots, a_n\in U^\ast$, $n\in\N$ there exists a characteristic
function $\phi_{\xi_{k,a_1,\dots,a_n}}$ of a probability measure $\xi_{k,a_1,\dots,a_n}$
on $\Borel(\R^n)$ such that
\begin{align}\label{eq.finiteidchar}
  \phi_{\mu\circ \pi_{a_1,\dots, a_n}^{-1}}(t)=\big(\phi_{\xi_{k,a_1,\dots, a_n}}(t)\big)^k
  \fa  t\in\R^n.
\end{align}

It follows that every infinitely divisible cylindrical probability measure $\mu$ is also
weakly infinitely divisible since for each $t=(t_1,\dots, t_n)\in\R^n$ we have
\begin{align*}
  \phi_{\mu\circ\pi_{a_1,\dots, a_n}^{-1}}(t)
  & =\phi_{\mu}(t_1 a_1+\dots +t_n a_n)\\
  &=\big(\phi_{\mu_k}(t_1 a_1+\dots +t_n a_n)\big)^k\\
  &=\left(\phi_{\mu_k\circ\pi_{a_1,\dots, a_n}^{-1}}(t)\right)^k
\end{align*}
for all $a_1,\dots, a_n\in U^\ast$ and all $k\in\N$. We will later see, that the converse
is also true, i.e. that the concepts of Definitions \ref{de.infdivcyl} and
\ref{de.weaklyinfdivcyl} coincide.

If $\mu$ is a weakly infinitely divisible cylindrical measure then $\mu\circ a^{-1}$ is
an infinitely divisible measure in $\Borel(\R)$ and thus,
\begin{align}\label{eq.charone}
  \phi_{\mu}(a)
  &=\phi_{\mu\circ a^{-1}}(1)\notag\\
  &= \exp\left(im_a-\tfrac{1}{2}r_a^2+\int_{\R} \left(e^{i s}-1-i s \1_{B_{\R}}(s)\right)\, \eta_a(ds)\right)
\end{align}
for some constants $m_a\in\R$, $r_a\ge 0$ and a L{\'e}vy measure $\eta_a$ on
$\Borel(\R)$. For infinitely divisible cylindrical measures, this representation can be
significantly improved as we have shown in Applebaum and Riedle \cite{DaveMarkus}. The
same prove establishes the result for weakly infinitely divisible cylindrical measures in
the following theorem.
\begin{theorem}\label{th.cyllevymeasure}
Let $\mu$ be a weakly infinitely divisible cylindrical probability measure on $\Z(U)$.
Then its characteristic function $\phi_\mu:U^\ast\to \C$ is given by
\begin{align}\label{eq.charth}
 & \phi_{\mu}(a)\\
 &\quad  = \exp\left(iw(a)-\tfrac{1}{2}q(a)+\int_U \left(e^{i\scapro{u}{a}}-1-i\scapro{u}{a}
   \1_{B_{\R}}(\scapro{u}{a})\right)   \, \nu (du)\right),\notag
\end{align}
where $w:U^\ast\to\R$ is a mapping, $q:U^\ast\to \R$ is a quadratic form and $\nu$ is a
cylindrical measure on $\Z(U)$ such that $\nu\circ \pi^{-1}_{a_1,\dots, a_n}$ is the
L{\'e}vy measure on $\Borel(\R^n)$ of $\mu\circ\pi_{a_1,\dots,a_n}^{-1}$ for all
$a_1,\dots,a_n\in U^\ast$, $n\in\N$.
\end{theorem}

It is natural to denote the measure $\nu$ appearing in \eqref{eq.charth} as a {\em cylindrical
L{\'e}vy measure} as we do in the following definition. However, it turns out that it is sufficient  to require only
that the image measures under all one-dimensional linear projections to $\R$ are L{\'e}vy measures
and it is not necessary  to consider the image measures under all linear projections to $\R^n$ for all finite dimensions
$n$.
\begin{definition}
  A cylindrical measure $\nu:\Z(U)\to [0,\infty]$ is called a {\em cylindrical L{\'e}vy
  measure} if $\nu\circ a^{-1}$ is a L{\'e}vy measure on $\Borel(\R)$ for all $a\in U^\ast$.
\end{definition}
From \eqref{eq.charth} we can easily derive a representation of the characteristic
function $\phi_{\mu}$ of a weakly infinitely divisible cylindrical probability measure
$\mu$ for an arbitrary truncation function $h$. Since $h=\Id$ on $D(h)$ one can define
\begin{align*}
p:U^\ast\to \R, \quad  p(a):=w(a)
  +\int_{U}\big(h(\scapro{u}{a})-\scapro{u}{a}\1_{B_{\R}}(\scapro{u}{a})\big)\,\nu(du).
\end{align*}
It follows from \eqref{eq.charth} that
\begin{align}\label{eq.charmu}
 \phi_\mu(a)=
  \exp\left(ip(a)-\tfrac{1}{2}q(a)+\int_U \psi_h(\scapro{u}{a})\,
  \nu(du)\right),
\end{align}
where the kernel function $\psi_h$ is defined by
\begin{align*}
  \psi_h:\R\to \C, \qquad \psi_h(t):=e^{it}-1-ih(t)
\end{align*}
for an arbitrary  truncation function $h$.
\begin{definition}\label{de.cylchar}
  Let $h$ be an truncation function and let $\mu$ be a weakly infinitely divisible cylindrical probability measure on
$\Z(U)$ with characteristic function \eqref{eq.charmu}. Then we call the triplet $(p, q,
\nu)_h$ the {\em cylindrical characteristics of $\mu$}.
\end{definition}
Analogously to the one-dimensional situation after Definition \ref{de.truncation} one can
convert the cylindrical characteristcs $(p,q,\nu)_h$ into $(p^\prime,q,\nu)_{h^\prime}$
if $h^\prime$ is another truncation function.

It follows from \eqref{eq.charmu} that the
characteristic function $\phi_{\mu\circ a^{-1}}$ of the probability measure $\mu\circ a^{-1}$ on
$\Borel(\R)$ is for all $t\in\R$ given by
\begin{align}\label{eq.charmua}
  \phi_{\mu\circ a^{-1}}(t)
  &=\phi_{\mu}(at)\notag\\
  &= \exp\left(ip(at)-\tfrac{1}{2}q(a)t^2+\int_{\R} \psi_h(st) \, (\nu\circ a^{-1})(ds)\right).
\end{align}
This representation of $\phi_{\mu\circ a^{-1}}$ does not coincide with the representation
\eqref{eq.charmuone} because the functions $\widetilde{\psi_h}$ and $\psi_h$ do not
coincide. Thus, we can not directly read out the characteristics of $\mu\circ a^{-1}$
from \eqref{eq.charmua}.
\begin{lemma}\label{le.characteristicsmua}
Let $\mu$ be a weakly infinitely divisible cylindrical probability measure on $\Z(U)$
with cylindrical characteristics $(p,q,\nu)_h$ for a truncation function $h$.
 Then $\mu\circ a^{-1}$ has the characteristics $(p(a),q(a),\nu\circ
a^{-1})_h$ for all $a\in U^\ast$.
\end{lemma}
\begin{proof}
As above we can rewrite the characteristic function $\phi_{\mu\circ a^{-1}}$ in the form \eqref{eq.charmua} for the given
truncation function $h$. In order to write $\phi_{\mu\circ a^{-1}}$ in the standard form \eqref{eq.charmuone},
we introduce the function $\widetilde{p}:U^\ast\times \R\to \R$ defined by
\begin{align*}
 \widetilde{p}(a,t):=
 \begin{cases}  p(at)+ \int_{\R}\big(t\,h(s)-h(st)\big)\,
  (\nu\circ a^{-1})(ds), & \text{if }t\neq 0,\\
  0 ,&\text{if }t=0. \end{cases}
\end{align*}
Note, that the integral is well defined  because for each $t\neq 0$ we have
\begin{align*}
 t\,h(s)-h(st)=0 \qquad\text{for all }s\in D(h)\cap \tfrac{1}{t}D(h)
\end{align*}
and because $h$ is bounded.  By defining the function
\begin{align*}
  \widetilde{\psi}_h:\R\times\R\to \C,
  \qquad \widetilde{\psi}_h(s,t)=e^{ist}-1-it h(s),
\end{align*}
we can rewrite the characteristic function \eqref{eq.charmua} of $\mu\circ a^{-1}$ for
all $t\in\R$:
\begin{align*}
  \phi_{\mu \circ a^{-1}}(t)&=
  \exp\left(i\widetilde{p}(a,t)-\tfrac{1}{2}q(a)t^2 + \int_{\R}\widetilde{\psi}_h(s,t)\,(\nu\circ
  a^{-1})(ds)\right).
\end{align*}
By Theorem \ref{th.cyllevymeasure} the L{\'e}vy measure of the infinitely divisible probability measure $\mu\circ a^{-1}$  is
given by $\nu\circ a^{-1}$ for each $a\in U^\ast$.
Thus, there exist some constants $m_a\in\R$ and $r_a\ge 0$ such that $(m_a,r_a,\nu\circ
a^{-1})_h$ is the characteristics of $\mu\circ a^{-1}$. For all $t\in\R$ it follows that
  \begin{align*}
    \phi_{\mu\circ a^{-1}}(t)
    &= \exp\left(i\widetilde{p}(a,t)-\tfrac{1}{2}q(a) t^2+\int_{\R} \widetilde{\psi}_h(s, t)\, (\nu\circ
    a^{-1})(ds)\right)\\
    &= \exp\left(i m_a t-\tfrac{1}{2}r_a^2t^2+\int_{\R} \widetilde{\psi}_h(s,t)\, (\nu\circ
    a^{-1})(ds)\right),
  \end{align*}
 which results in $\widetilde{p}(a,t)=m_a t=\widetilde{p}(a,1)t=p(a)t$.
 Consequently, we have
\begin{align*}
  \phi_{\mu\circ a^{-1}}(t)
    = \exp\left(i p(a)t-\tfrac{1}{2}q(a)t^2+\int_{\R} \widetilde{\psi}_h(s, t)\, (\nu\circ
    a^{-1})(ds)\right),
\end{align*}
which completes the proof.
\end{proof}

Recalling the L{\'e}vy-Khintchine decomposition for infinitely divisible measures we could expect from \eqref{eq.charmu} that
\begin{align*}
   a\mapsto\exp\left(ip(a)\right),
  \qquad a\mapsto\exp\left(\int_{U} \psi_h(\scapro{u}{a})\,\nu(du)\right)
\end{align*}
are characteristic functions of cylindrical measures on $\Z(U)$, respectively. But the
following example shows that  we can not separate the drift part $p$ and the integral term with respect to the
cylindrical L{\'e}vy measure $\nu$ in order to obtain cylindrical measures.
\begin{example}\label{ex.poisson1}
Let $\ell:U^\ast\to\R$ be a linear but not necessarily a continuous functional and
$\lambda>0$ a constant. We will see later in Example \ref{ex.poissoncont} that
\begin{align*}
  \phi:U^\ast\to\C, \qquad
  \phi(a):=\exp\left(\lambda\left(e^{i \ell(a)}-1\right)\right)
\end{align*}
is the characteristic function of an infinitely divisible cylindrical probability
measure. In order to write $\phi$ in the form \eqref{eq.charmu} let $\nu$ be the
cylindrical measure on $\Z(U)$ defined by
\begin{align*}
  \nu(Z(a_1,\dots, a_n;B)):=\begin{cases}
   \lambda, & \text{if } (\ell(a_1),\dots, \ell(a_n))\in B, \\
   0, & \text{else,}
   \end{cases}
\end{align*}
for every $a_1,\dots, a_n\in U^\ast$, $B\in \Borel(\R^n)$ and $n\in\N$.
 Then we can represent $\phi$ by
\begin{align*}
  \phi(a)=\exp\left(ip(a)+\int_U \psi_h(\scapro{u}{a})\,\nu(du)\right),
\end{align*}
where $p(a):=\lambda h(\ell(a))$. Since $a\mapsto \exp(ip(a))$ is not positive-definite
in general there does not exist a cylindrical measure with this function as its
characteristic function.
\end{example}

Example \ref{ex.poisson1} leads us to the insight that some necessary conditions
guaranteeing the existence of an infinitely divisible cylindrical probability measure
with cylindrical characteristics  $(p,0,\nu)$ rely on the interplay of the entries $p$
and $\nu$. The following result gives some properties of the entries $p$, $q$ and $\nu$
of the cylindrical characteristics, respectively, but also the interplay of $p$ and
$\nu$.
\begin{lemma}\label{le.necc}
Let $\mu$ be a weakly infinitely divisible cylindrical probability measure on $\Z(U)$
with cylindrical characteristics $(p,q,\nu)_h$ for a continuous truncation function $h$.
It follows that:
\begin{enumerate}
  \item[{\rm (a)}] $\displaystyle
    a\mapsto \kappa(a):= - \left(i p(a)+ \int_{U} \psi_h(\scapro{u}{a})\,\nu(du)\right)$
    is negative-definite.
  \item[{\rm (b)}]   for every sequence $a_n\to a$ in a finite dimensional subspace $V\subseteq U^\ast$
   equipped with $\norm{\cdot}_{U^\ast}$ we have:
          \begin{enumerate}
       \item[{\rm (i)}]
           $ p(a_n)\to p(a)$;
         \item[{\rm (ii)}] $q(a_n)\to q(a)$.
         \item[{\rm (iii)}]$ \left(\abs{s}^2\wedge 1\right)\, (\nu\circ a_n^{-1})(ds)
           \to \left(\abs{s}^2\wedge 1\right)\, (\nu\circ a^{-1})(ds)$ weakly;
          \end{enumerate}
\end{enumerate}
\end{lemma}
\begin{proof}
(a):
Let $Z$ be a cylindrical random variable on a probability space $(\Omega,\A,P)$ with cylindrical distribution
$\mu$. As in Theorem 3.9 in \cite{DaveMarkus} it follows that there exist two  cylindrical random variables $W$ and $X$ such
that $Z=W+X$ $P$-a.s. where the cylindrical distributions $\mu_1$ of $W$ and $\mu_2$ of $X$ have the characteristic functions
$\phi_1$ and $\phi_2$ given by
\begin{align*}
  \phi_1(a):=\exp(-\tfrac{1}{2}q(a)),\qquad \phi_2(a):=\exp(-\kappa(a)).
\end{align*}
For fixed $a_1,\dots, a_n\in U^\ast$  the $\R^n$-valued random variable $(Za_1,\dots, Za_n)$ is
infinitely divisible since $\mu$ is assumed to be weakly infinitely divisible and the $\R^n$-valued random
variable $(Wa_1,\dots,Wa_n)$ is also infinitely divisible as it is Gaussian. Thus, the $\R^n$-valued random variable $(Xa_1,\dots, Xa_n)$ is infinitely divisible, that is the cylindrical measure $\mu_2$ is weakly infinitely divisible.

We show (a) by applying Schoenberg's correspondence, see
\cite[Property(h), p.192]{Vaketal}, for which we have to show that $a\mapsto
\exp(-\tfrac{1}{k}\kappa(a))$ is positive-definite for all $k\in\N$ and that $\kappa$ is
Hermitian, i.e. $\overline{\kappa(-a)}=\kappa(a)$ for all $a\in U^\ast$. To prove
positive-definiteness, fix $k\in\N$, $a_1,\dots, a_n\in U^\ast$ and $z_1,\dots, z_n\in
\C$ and let $e_i$ denote the $i$-th unit vector in $\R^n$. Since $\mu_2$ is weakly infinitely divisible
there exists a characteristic function $\phi_{\xi_{k,a_1,\dots,a_n}}$ of a probability
measure $\xi_{k,a_1,\dots, a_n}$ on $\Borel(\R^n)$  such that
\begin{align*}
 \phi_{\mu_2\circ \pi_{a_1,\dots,a_n}^{-1}}(t)=\left(\phi_{\xi_{k,a_1,\dots,a_n}}(t)\right)^k\fa
  t\in\R^n.
\end{align*}
Consequently, we have
\begin{align*}
  \sum_{i,j=1}^n z_i \bar{z}_j\exp\left(-\tfrac{1}{k}\kappa(a_i-a_j)\right)
  &= \sum_{i,j=1}^n z_i\bar{z}_j\left(\phi_{\mu_2}(a_i-a_j)\right)^{1/k}\\
    &= \sum_{i,j=1}^n z_i\bar{z}_j\left(\phi_{\mu_2\circ\pi^{-1}_{a_1,\dots,a_n}}(e_i-e_j)\right)^{1/k}\\
  &=\sum_{i,j=1}^n z_i\bar{z}_j \phi_{\xi_{k,a_1,\dots, a_n}}(e_i-e_j)\\
  &\ge 0,
\end{align*}
where the last line follows from the fact that $\phi_{\xi_{k,a_1,\dots, a_n}}$ is a characteristic function on $\R^n$.

Next, we want show that $\kappa $ is Hermitian. Since rewriting the characteristic function of $\mu_2$ for different
 truncation functions does not effect the function $\kappa$
we can fix $h(s)=s\1_{B_{\R}}(s)$ for $s\in\R$ which yields
$\widetilde{\psi}_h(-s,t)=\widetilde{\psi}_h(s,-t)$ for all $s,t\in\R$. By Lemma
\ref{le.characteristicsmua} we obtain for all $t\in\R$ that
 \begin{align*}
   \phi_{\mu_2\circ a^{-1}}(t)
   &= \phi_{\mu_2\circ (-a)^{-1}}(-t)\\
   &=  \exp\left(i p(-a)(-t)+\int_{\R} \widetilde{\psi}_h(s,-t)\, (\nu\circ(-a)^{-1})(ds)\right)\\
   &= \exp \left(i p(-a)(-t)+\int_{\R}\widetilde{\psi}_h(-s,-t)\,(\nu\circ a^{-1})(ds)\right)\\
   &= \exp \left(i p(-a)(-t)+\int_{\R}\widetilde{\psi}_h(s,t)\,(\nu\circ
   a^{-1})(ds)\right),
 \end{align*}
 which implies $p(-a)=-p(a)$. It follows that
 \begin{align*}
   \overline{\kappa(-a)}&=\overline{-i p(-a)}-\int_{U}\overline{\psi_h(\scapro{u}{-a})}\,\nu(du)\\
   &=-i  p(a) -\int_U  \psi_h(\scapro{u}{a})\,\nu(du)\\
   &=\kappa(a)
 \end{align*}
 for all $a\in U^\ast$, which completes the proof of (a).

 To see (b) let $a_n\to a$ in a finite-dimensional subspace $V\subseteq U^\ast$ and let the truncation function $h$ be
 continuous. Then Bochner's theorem implies  that
 \begin{align}\label{eq.charconv}
   \lim_{n\to\infty}\phi_{\mu\circ a_n^{-1}}(t)
    = \lim_{n\to\infty}\phi_{\mu}(ta_n)
    = \phi_{\mu}(ta)
    =\phi_{\mu\circ a^{-1}}(t)
 \end{align}
for all $t\in \R$. By Lemma \ref{le.characteristicsmua} the measures $\mu\circ a_n^{-1}$
are infinitely divisible with characteristics $(p(a_n), q(a_n),\nu\circ a_n^{-1})$. It
follows from \eqref{eq.charconv} that the infinitely divisible measures with
characteristics $(p(a_n),q(a_n),\nu\circ a_n^{-1})$ converge weakly to $\mu\circ a^{-1}$
which has the characteristics $(p(a),q(a),\nu\circ a^{-1})$. Applying Theorem VII.2.9 and
Remark VII.2.10 (p.396) in Jacod and Shiryaev which characterises the weak convergence of
infinitely divisible measures in terms of their characteristics implies $p(a_n)\to p(a)$
and
\begin{align*}
  &q(a_n)\,\delta_0(ds) + \big(\abs{s}^2\wedge 1\big)\,(\nu\circ a_n^{-1})(ds)\\
  &\hspace*{2cm}
   \to q(a)\,\delta_0(ds) + \big(\abs{s}^2\wedge 1\big)\,(\nu\circ a^{-1})(ds)
   \qquad\text{weakly}.
\end{align*}
But since $q$ is a quadratic form and therefore it is continuous on a finite-di\-men\-sional space we have
$q(a_n)\to q(a)$ which is property (ii) and which results in (iii).

\end{proof}

\begin{theorem}\label{th.characteristics}
Let $\nu:\Z(U)\to [0,\infty]$ be a given set function and $p$, $q:U^\ast\to \R$ be given
functions and let $h$ be a continuous truncation function. Then the following are
equivalent:
\begin{enumerate}
  \item[{\rm(a)}] there exists an infinitely divisible cylindrical probability measure $\mu$ with
  cylindrical characteristics $(p,q,\nu)_h$;
  \item[{\rm (b)}] the following is satisfied:
    \begin{enumerate}
       \item[{\rm (1)}] $p(0)=0$ and $p(a_n)\to p(a)$ for every sequence $a_n\to a$ in a finite dimensional subspace $V\subseteq U^\ast$
      equipped with $\norm{\cdot}_{U^\ast}$;
       \item[{\rm (2)}] $q:U^\ast\to\R$ is a quadratic form;
       \item[{\rm (3)}] $\nu$ is a cylindrical L{\'e}vy measure;
       \item[{\rm (4)}] the function         
       \begin{align*}
 \qquad   a\mapsto \kappa(a):= - \left(i p(a)+ \int_{U} \psi_h(\scapro{u}{a})\,\nu(du)\right)
 \end{align*}
    is negative-definite.
    \end{enumerate}
\end{enumerate}
In this situation, the characteristic function of $\mu$ is given by
\begin{align*}
  \phi_\mu:U^\ast\to \C,
  \qquad \phi_\mu(a)=
  \exp\left(ip(a)-\tfrac{1}{2}q(a)+\int_U \psi_h(\scapro{u}{a})\, \nu(du)\right)
\end{align*}
and $\mu=\mu_1\ast\mu_2$ where $\mu_1$ and $\mu_2$ are cylindrical probability measures
with characteristic functions $\phi_{\mu_1}(a)=\exp(-\tfrac{1}{2}q(a))$ and
$\phi_{\mu_2}(a)=\exp(-\kappa(a))$.
\end{theorem}
\begin{proof}
(a)$\Rightarrow$(b):  The properties (2) and (3) are stated in Theorem
\ref{th.cyllevymeasure} and the properties (1) and (4) are derived in Lemma
\ref{le.necc}. The property $p(0)=0$ is an immediate consequence of Bochner's theorem as
is the fact that $q(0)=0$.

(b)$\Rightarrow$(a): Property (2) implies that
\begin{align*}
  \phi_1:U^\ast\to \C, \qquad \phi_1(a):=e^{-\tfrac{1}{2}q(a)}
\end{align*}
is the characteristic function of a Gaussian cylindrical probability measure $\mu_1$, see
\cite{riedlewiener} or \cite[p.393]{Vaketal}. Since  also $\frac{1}{k}q$ is a quadratic
form for every $k\in\N$ it follows that $(\phi_1)^{1/k}$ is the characteristic function
of a cylindrical measure  which verifies that $\mu_1$ is infinitely divisible. Thus, we are left to establish that
\begin{align*}
  \phi_2:U^\ast\to \C, \qquad \phi_2(a):=e^{-\kappa(a)}
\end{align*}
is the characteristic function of an infinitely divisible cylindrical measure.

For that purpose we show that the functions
\begin{align*}
  \phi_k:U^\ast\to \C,\qquad
  \phi_k(a):=\exp\left(\frac{1}{k}\left(ip(a)+\int_{U}\psi_h(\scapro{u}{a})\,\nu(du)\right)\right)
\end{align*}
are the characteristic function of a cylindrical probability measure for each $k\in \N$ .
The case $k=1$ shows that there exists a cylindrical measure $\mu$ with characteristic
function $\phi_1$ and the cases $k\ge 1$ show that $\mu$ is infinitely divisible. Note
firstly, that the integral in the definition of $\phi_k$ exists and is finite because of
condition (3).

Obviously, $\phi_k(0)=1$ by (1) and (3). Property (4) implies by  the Schoenberg's correspondence
for functions on Banach spaces (property (h), p.192 in \cite{Vaketal}) that $\phi_k$ is
positive-definite. In order to show the last condition of Bochner's theorem let
$V\subseteq U^\ast$ be a finite-dimensional subspace, say $V=\text{span}\{b_1,\dots,
b_d\}$ for $b_1,\dots, b_d\in U^\ast$ and $a_n\to a_0$ in $V$. Then $(U,\Z(U,\{b_1,\dots,
b_d\}),\nu)$ is a measure space. Let $f:\R\to\R$ be a bounded continuous function and
define
\begin{align*}
  g_n:U\to\R,\qquad
   g_n(u):=f(\scapro{u}{a_n})\,\left(\abs{\scapro{u}{a_n}}^2\wedge 1\right)
\end{align*}
for $n\in \N\cup\{0\}$. It follows by (3) that each $g_n\in L_\nu^1(U,\Z(U,\{b_1,\dots,
b_d\}))$ and
\begin{align*}
  \abs{g_n(u)}\le \norm{f}_\infty (1+c) \left( \abs{\scapro{u}{a_0}}^2\wedge 1\right)
\end{align*}
for a constant $c>0$. Lebesgue's theorem of dominated convergence implies that
\begin{align*}
  \lim_{n\to\infty} \int_{U} g_n(u)\,\nu(du)=\int_U g_0(u)\,\nu(du),
\end{align*}
which shows that
\begin{align}\label{eq.auxweak}
   \left(\abs{s}^2\wedge 1\right)\, (\nu\circ a_n^{-1})(ds)
           \to \left(\abs{s}^2\wedge 1\right)\, (\nu\circ a^{-1})(ds)
            \text{ weakly}.
\end{align}
Condition (3) guarantees for each $a\in U^\ast$, that
  \begin{align*}
\phi_{\mu_a}:\R\to\C, \qquad    \phi_{\mu_a}(t)=\exp\left(i p(a)t + \int_{\R}
\widetilde{\psi}_h(s,t)\,(\nu\circ a^{-1})(ds)\right)
  \end{align*}
is the characteristic function of an infinitely divisible probability measure, say
$\mu_a$ on $\Borel(\R)$ with characteristics $(p(a),0,\nu\circ a^{-1})_h$. Then condition
(1) together with the weak convergence in \eqref{eq.auxweak}  imply by Theorem VII.2.9
and Remark VII.2.10 (p.396) in \cite{JacodShir} that $\phi_{\mu_{a_n}}(t)\to
\phi_{\mu_{a_0}}(t)$ for all $t\in\R$. Because $\phi_k(a)=\left(\phi_{\mu_a}(1)\right)^{1/k}$
for all $a\in U^\ast$ and $k\in\N$ the functions $\phi_k$ are verified as continuous on
every finite-dimensional subspace which is the last condition in Bochner's theorem.

The remaining part follows directly from the proof of (b).
\end{proof}

\begin{example}\label{ex.poissoncont}
Now we can show that  the function $\phi$ in Example \ref{ex.poisson1} is in fact the
characteristic function of an infinitely divisible cylindrical measure. The linearity of $\ell$ yields that
the mapping $a\mapsto p(a)=\lambda h(\ell(a))$ is continuous on each finite dimensional subspace of $U^\ast$ if the
truncation function $h$ is continuous. The measure $\nu$ satisfies $\nu\circ a^{-1}=\lambda\delta_{\ell(a)}$ for each $a\in U^\ast$
and is therefore a cylindrical L{\'e}vy measure.  Since
\eqref{eq.psinegative} yields that
\begin{align*}
f:\R\to \C,\qquad  f(t):=- \lambda\left(e^{it}-1\right)
\end{align*}
is a  negative-definite function it follows for  $z_1,\dots ,z_n\in \C$, $a_1,\dots,
a_n\in U^\ast$ that
\begin{align*}
&  \sum_{i,j=1}^n z_i\bar{z}_j \kappa(a_i-a_j) =
 \sum_{i,j=1}^n z_i\bar{z}_j f\big(\ell(a_i)-\ell(a_j)\big)
\le 0.
\end{align*}
Thus, the map $\kappa$ is negative-definite  which proves the claim
due to Theorem \ref{th.characteristics}.
\end{example}

For a given cylindrical L{\'e}vy measure $\nu$ there does not exist in general an
infinitely divisible cylindrical probability measure with cylindrical characteristics
$(0,0,\nu)$, see Example \ref{ex.poisson1}. But one might be able to construct a function $p:U^\ast\to \R$ such that
there exists a cylindrical probability measure with cylindrical characteristics
$(p,0,\nu)$.

The following example shows the construction of the function $p$  for a given cylindrical
L{\'e}vy measure $\nu$ with weak second moments. In Section \ref{se.measure} we consider
the case if the cylindrical L{\'e}vy measure extends to a $\sigma$-finite measure on
$\Borel(U)$.
\begin{example}
 Let $\nu$ be a cylindrical L{\'e}vy measure which satisfies
 \begin{align*}
   \int_{U}\abs{\scapro{u}{a}}^2\,\nu(du)<\infty\qquad\text{for all }a\in
   U^\ast.
 \end{align*}
The existence of the weak second moments enables us to define
\begin{align*}
p:U^\ast\to\R, \qquad  p(a):=\int_U \left(h(\scapro{u}{a})-
\scapro{u}{a}\right)\,\nu(du)
\end{align*}
for a continuous truncation function $h$. With a careful analysis similar to the one in the proof of Theorem \ref{th.Levymeasure1} it can be shown that $p$ is continuous on every finite-dimensional subspace of $U^\ast$. From
\eqref{eq.psinegative} it follows that
\begin{align*}
  f:\R\to\C  \qquad f(t):=-\left(e^{it}-1- it\right)
\end{align*}
is negative-definite. For  $z_1,\dots ,z_n\in \C$, $a_1,\dots, a_n\in U^\ast$ we have:
\begin{align*}
&  \sum_{i,j=1}^n -z_i\bar{z}_j\left(ip(a_i-a_j)+ \int_U
  \psi_h(\scapro{u}{a_i-a_j})\,\nu(du)\right)\\
&\qquad  = \int_U \sum_{i,j=1}^n z_i\bar{z}_j
f(\scapro{u}{a_i}-\scapro{u}{a_j})\,\nu(du)\le 0.
\end{align*}
Theorem \ref{th.characteristics} shows that there exists an infinitely divisible
cylindrical measure with cylindrical characteristics $(p,0,\nu)_h$.
\end{example}

We finish this section with establishing that our two Definitions \ref{de.infdivcyl} and \ref{de.weaklyinfdivcyl} of infinitely divisibility for cylindrical measures coincide. In particular, this result enables us to show that a Radon measure is already
infinitely divisible if all its finite dimensional projections are infinitely divisible.
\begin{theorem}\hfill
\begin{enumerate}
  \item[{\rm (a)}] A cylindrical probability measure $\mu$ on $\Z(U)$ is infinitely divisible
  if and only if it is weakly infinitely divisible.
    \item[{\rm (b)}] A Radon probability measure $\mu$ on $\Borel(U)$ is infinitely divisible if and only if
  \begin{align*}
    &\mu\circ\pi_{a_1,\dots, a_n}^{-1} \text{ is an infinitely divisible probability measure for all
    }\\
    &\qquad a_1,\dots, a_n\in U^\ast, n\in\N.
  \end{align*}
\end{enumerate}
\end{theorem}
\begin{proof}
  (a) If $\mu$ is weakly infinitely divisible then Theorem \ref{th.cyllevymeasure} and
 Lemma \ref{le.necc} guarantee that the cylindrical characteristics of $\mu$ satisfies
  the conditions in Theorem  \ref{th.characteristics}.

  (b) Let the Radon measure $\mu$ satisfy that all its finite dimensional projections are
  infinitely divisible. Then the restriction of $\mu$ to $\Z(U)$ is a weakly infinitely divisible cylindrical measure
  and  it follows from (a) that for each $k\in\N$ there exists a cylindrical probability
  measure $\mu_k$ such that $\mu=\mu_k^{\ast k}$. Theorem 1 in  \cite{Ros} implies that
  there exists $\ell$ in the algebraic dual $U^{\ast \prime}$ of $U^\ast$ such that
  $\mu_k\ast\delta_{\ell}$ is a Radon probability measure where
  \begin{align*}
    \delta_{\ell}(Z):=\begin{cases}
      1, &\text{if } (\ell(a_1),\dots,\ell(a_n))\in B,\\
      0, &\text{otherwise}
    \end{cases}
  \end{align*}
for every $Z:=Z(a_1,\dots,a_n;B)\in\Z(U)$.   Since
  \begin{align*}
    \mu\ast\delta_\ell^{\ast k}=\mu_k^{\ast k}\ast\delta_\ell^{\ast k}
     =(\mu_k\ast \delta_\ell)^{\ast k}
  \end{align*}
  and the right hand side is Radon it follows from \cite[Prop.7.14.50]{Bogachev} that $\delta_\ell^{\ast k}$ is a Radon
  probability measure which implies $\ell\in U$ by considering the characteristic functions.
  Since then $\mu_k\ast\delta_{\ell}$ and $\delta_{\ell}$ are Radon probability measures a
  further application of \cite[Prop.7.14.50]{Bogachev} implies that $\mu_k$ is a Radon
  probability measure, which shows that $\mu$ is  an infinitely divisible Radon measure.
\end{proof}

\section{Continuous infinitely divisible cylindrical measures}

Continuity of cylindrical measures is defined with respect to an arbitrary vector
topology $\O$ in $U^\ast$. We assume here that the topological space $(U^\ast, \O)$
satisfies the first countability axiom, that is that every neighborhood system of every
point in $U^\ast$ has a countable local base. In such spaces, convergence is equivalent
to sequential convergence. In particular, $U^\ast$ equipped with the norm topology
satisfies the first countability axiom.
\begin{definition}
  A cylindrical probability measure $\mu$ on $\Z(U)$ is called {\em $\O$-con\-tin\-uous} if for
  each $\epsilon>0$ there exists a neighborhood $N$ of $0$ such that
  \begin{align*}
    \mu\left(\{u\in U:\abs{\scapro{u}{a}}\ge 1 \}\right)\le
     \epsilon
  \end{align*}
for all $a\in N$.
If $\O$ is the norm topology we say {\em $\mu$ is continuous}.
\end{definition}
A cylindrical probability measure $\mu$ is $\O$-continuous if and only if its
characteristic function $\phi_{\mu}:U^\ast\to \C$ is continuous in the topology $\O$, see
\cite[Th.II.3.1]{Schwartz}. This enables us to derive the following criteria:
\begin{lemma}\label{le.continuouscylinfdiv}
  Let $\mu$ be an infinitely divisible cylindrical probability measure on $\Z(U)$ with cylindrical characteristics $(p,q,\nu)_h$ for
  a continuous truncation function $h$.
  Then the following are equivalent:
  \begin{enumerate}
    \item[{\rm (a)}] $\mu$ is $\O$-continuous;
    \item[{\rm (b)}] for every sequence $a_n\to a$ in $(U^\ast,\O)$ we have:
      \begin{enumerate}
          \item[{\rm (i)}] $ p(a_n)\to p(a)$;
         \item[{\rm (ii)}]$ q(a_n)\delta_0(ds)+ \left(\abs{s}^2\wedge 1\right)\, (\nu\circ
         a_n^{-1})(ds)$\\
            \hspace*{1cm}$\displaystyle\to q(a)\delta_0(ds)+ \left(\abs{s}^2\wedge 1\right)\, (\nu\circ a^{-1})(ds)$
             weakly.
      \end{enumerate}
  \end{enumerate}
\end{lemma}
\begin{proof}
  The cylindrical measure $\mu$ is $\O$-continuous if and only if its characteristic
  function $\phi_\mu:U^\ast\to \C$ is continuous in $(U^\ast,\O)$, or equivalently, that
  $\phi_\mu:U^\ast\to \C$ is sequentially continuous. It follows as in the proof of Theorem  \ref{th.characteristics} that:
  \begin{align*}
  &  \phi_{\mu}(a_n)\to \phi_\mu(a)\quad\text{for all sequences }a_n\to a \text{ in $(U^\ast,\O)$}\\
   &\quad \Longleftrightarrow\;
     \phi_{\mu\circ a_n^{-1}}(t)\to \phi_{\mu\circ a^{-1}}(t)
     \quad\text{for all sequences $a_n\to a$ in $(U^\ast,\O)$,} t\in\R.
  \end{align*}
By applying Theorem VII.2.9 and Remark VII.2.10  in \cite{JacodShir} and Lemma \ref{le.characteristicsmua} the right
hand side is equivalent to the conditions (i) and (ii) in (b) which completes the proof.
\end{proof}

In Lemma \ref{le.continuouscylinfdiv} it does not follow from (a) that we can consider
separately the quadratic form $q$ and the term depending on the  cylindrical L{\'e}vy measure $\nu$ in condition
(b). This is due to the well known fact, that a sequence of infinitely divisible measures
on $\Borel(\R)$ can converge weakly such that the small jumps contribute to the Gaussian
part in the limit. But since an infinitely divisible cylindrical measure $\mu$ is the convolution
of two other infinitely divisible cylindrical measures it is of interest whether the continuity of $\mu$ is
inherited by the convolution cylindrical measures.
\begin{definition}
  An infinitely divisible cylindrical probability measure with cylindrical characteristics $(p,q,\nu)_h$
  is called {\em regularly $\O$-continuous} if the infinitely divisible cylindrical probability measures with
  cylindrical characteristics $(0,q,0)_h$ and $(p,0,\nu)_h$ are $\O$-continuous.
\end{definition}

\begin{lemma}\label{le.regular}
Let $h$ be a continuous truncation function.
For an $\O$-con\-tin\-uous infinitely divisible cylindrical probability measure $\mu$ with cylindrical
characteristics $(p,q,\nu)_h$ the following are equivalent:
\begin{enumerate}
 \item[{\rm (a)}] $\mu$ is regularly $\O$-continuous;
 \item[{\rm (b)}] $q:U^\ast\to \R$ is continuous in $(U^\ast,\O)$;
 \item[{\rm (c)}] for every sequence $a_n\to a$ in $(U^\ast,\O)$ we have:
      \begin{enumerate}
     \item[{\rm (i)}] $p(a_n)\to p(a)$;
    \item[{\rm (ii)}]$\displaystyle \left(\abs{s}^2\wedge 1\right)\, (\nu\circ a_n^{-1})(ds)\to
           \left(\abs{s}^2\wedge 1\right)\, (\nu\circ a^{-1})(ds)\text{ weakly}$.
      \end{enumerate}
\end{enumerate}
\end{lemma}
\begin{proof}
Let $\phi_\mu$ be the characteristic function of $\mu$. Then $\phi_\mu=\phi_1\cdot\phi_2$
where $\phi_1$ is the characteristic function of the cylindrical measure $\mu_1$ with cylindrical characteristics $(0,q,0)_h$ and $\phi_2$ is the characteristic function of the cylindrical measure $\mu_2$ with cylindrical characteristics $(p,0,\nu)_h$.
   Since the characteristic function of an infinitely
   divisible measure does not vanish in any point it follows that continuity of $\phi_1$ and $\phi_\mu$ results in the continuity of $\phi_2$ and analogously if $\phi_2$ and $\phi_\mu$ are continuous then $\phi_1$ is continuous. Thus, $\mu$ is regularly $\O$-continuous if
   and only if either $\mu_1$ or $\mu_2$ is $\O$-continuous.

Applying Lemma \ref{le.continuouscylinfdiv} to $\mu_1$ shows the equivalence (a)$\Leftrightarrow$ (b) and
applying Lemma \ref{le.continuouscylinfdiv} to $\mu_2$ shows the equivalence (a)$\Leftrightarrow$ (c).
\end{proof}

\begin{remark}\label{re.cov}
If $U^\ast$ is equipped with the norm topology then (b) in Lemma
\ref{le.regular} can be replaced by
\begin{enumerate}
  \item[{\rm (b$^\prime$)}]   there exists a positive, symmetric operator $Q:U^\ast\to U^{\ast\ast}$ such that
  $q(a)=\scapro{a}{Qa}$ for
 all $a\in U^\ast$;
\end{enumerate}
\begin{proof}
According to Proposition IV.4.2 in \cite{Vaketal} there exist a probability space
$(\Omega,\A,P)$ and a cylindrical random variable $X:U^\ast\to L^0_P(\Omega,\A)$  with cylindrical distribution
$(0,q,0)$ and with characteristic function $a\mapsto\phi(a)=\exp(-\tfrac{1}{2}q(a))$.
If $q$ is continuous Proposition VI.5.1 in \cite{Vaketal} implies that the mapping $X:U^\ast\to L^0_P(\Omega,\A)$ is continuous. Consequently, it follows from Theorem 4.7 in \cite{DaveMarkus} that $(Qa)b:=E[(Xa)(Xb)]$ for $a,b\in U^\ast$ defines a positive, symmetric  operator $Q:U^\ast\to U^{\ast\ast}$. Obviously, it satisfies $q(a)=(Qa)a$ for each $a\in U^\ast$.
\end{proof}
\end{remark}

\begin{example}
Let $(\Omega, \A,P)$ be a probability space and let $L:=(L(t):\, t\ge 0)$ be a
cylindrical process, that is  the mappings $L(t):U^\ast\to L^0_P(\Omega,\A)$
are linear. In Applebaum and Riedle \cite{DaveMarkus} we call $L$ a {\em
cylindrical L{\'e}vy process} if
\begin{align*}
  \big((L(t)a_1,\dots, L(t)a_n):\,t\ge 0\big)
\end{align*}
is a L{\'e}vy process in $\R^n$ for all $a_1,\dots, a_n\in U^\ast$, $n\in\N$.
If $L$ is a cylindrical L{\'e}vy process we derive in \cite{DaveMarkus}  that
it can be decomposed according to
\begin{align*}
  L(t)=W(t)+Y(t)\qquad\text{for all $t\ge 0$ $P$-a.s.,}
\end{align*}
where $(W(t):\,t\ge 0)$ and $(Y(t):\,t\ge 0)$ are cylindrical processes. Their
characteristic functions are for all $a\in U^\ast$  given by
\begin{align*}
 \phi_{W(t)}(a):= E[\exp(iW(t)a)]=\exp\left(-\tfrac{1}{2}q(a)t\right)
\end{align*}
for a quadratic form $q:U^\ast\to \R$ and
\begin{align*}
 \phi_{Y(t)}(a):= E[\exp(iY(t)a)]=\exp\left(t\Big(ip(a)+\int_U \psi_h(\scapro{u}{a})\,\nu(du)\Big)\right)
\end{align*}
for a mapping $p:U^\ast\to\R$ and a cylindrical L{\'e}vy measure $\nu$.
Obviously, $(p,q,\nu)_h$ is the cylindrical characteristics of an infinitely divisible cylindrical measure $\mu$. If $\mu$ is regularly continuous, i.e. the cylindrical measures with
the characteristic functions $\phi_{W(1)}$ and $\phi_{L(1)}$ are continuous, it
follows that also the mappings
 \begin{align*}
   W(t):U^\ast\to L^0_P(\Omega,\A),
   \qquad
   Y(t):U^\ast\to L^0_P(\Omega,\A),
 \end{align*}
are continuous, see \cite[Prop.VI.5.1]{Vaketal}. Moreover, according to Remark
\ref{re.cov} the quadratic form $q$ is of the form $q(a)=\scapro{a}{Qa}$ for all $a\in
U^\ast$ and for a symmetric, positive operator $Q:U^\ast\to U^{\ast\ast}$. If
$Q(U^\ast)\subseteq U$ then $W$ is a cylindrical Wiener process in a strong sense as it
is usually considered in the literature, see Riedle \cite{riedlewiener}.
\end{example}

\section{L{\'e}vy measures on Banach spaces}\label{se.measure}

In this section we consider the  situation that the cylindrical L{\'e}vy measure $\nu$
extends to a $\sigma$-finite measure on $\Borel(U)$ which is also denoted by $\nu$.
The unit ball is denoted by $B_U:=\{u\in U:\,\norm{u}\le 1\}$.

\begin{theorem}\label{th.Levymeasure1}
Let $\nu$ be a cylindrical L{\'e}vy measure which extends to a $\sigma$-finite measure on
$\Borel(U)$ with $\nu(B_{U}^c)<\infty$. Then there exists a regularly continuous
infinitely divisible cylindrical probability measure $\mu$ with cylindrical
characteristics $(d_\nu,0,\nu)_h$, where
\begin{align*}
    d_{\nu}:U^\ast\to\R, \qquad
    d_{\nu}(a):=\int_U \big(h(\scapro{u}{a})-\scapro{u}{a}\1_{B_U}(u)\big)\, \nu(du).
  \end{align*}
\end{theorem}
\begin{proof}
  First we show that the integral in the definition of the function $d_\nu$ is well
  defined for every truncation function $h$. Choose a constant $c>0$ such that
\begin{align*}
  \{t\in\R:\,\abs{t}\le c\}\subseteq D(h)
\end{align*}
and define for every $a\in U^\ast$ the set
\begin{align*}
D(a):=\{v\in U:\, \abs{\scapro{v}{a}}\le c\}.
\end{align*}
For the integrand $f_a(u):=h(\scapro{u}{a})-\scapro{u}{a}\1_{B_U}(u)$ it follows for
every $u\in U$ that
\begin{align*}
f_a(u)\neq 0
 \;\Rightarrow\; u \in \big( D(a)\cap B_U^c\big) \cup \big( D^c(a)\cap B_U \big) \cup \big(D^c(a)\cap B_U^c \big).
\end{align*}
But on these three domains we obtain
\begin{align*}
  &\int_{D(a)\cap B_U^c}\abs{f_a(u)}\, \nu(du)
  \le \int_{B_U^c} c \,\nu(du) =c\, \nu(B_U^c)<\infty,
\end{align*}
and
\begin{align*}
  \int_{D^c(a)\cap B_U}\abs{f_a(u)} \,\nu(du)
  & \le  \int_{c < \abs{\scapro{u}{a}}\le \norm{a}}
    \abs{h(\scapro{u}{a})-\scapro{u}{a}}\,\nu(du)\\
  &  =  \int_{c < \abs{s}\le \norm{a}}\abs{h(s)-s}\,(\nu\circ a^{-1})(ds)\\
  & \le (\norm{h}_\infty + \norm{a})\,(\nu\circ a^{-1})(\{s\in\R: \abs{s}> c\})<\infty,
\end{align*}
because $\nu\circ a^{-1}$ is a L{\'e}vy measure on $\Borel(\R)$ and
\begin{align*}
  \int_{D^c(a)\cap B_U^c} \abs{f_a(u)}\,\nu(du)
 \le \norm{h}_\infty \int_{B_U^c}\,\nu(du)
  = \norm{h}_\infty \nu(B_U^c)<\infty.
\end{align*}
Now we choose the truncation function $h$ to be continuous and show by a similar
decomposition that $d_\nu$ is continuous. Let $a_n\to a$ in $U^\ast$ and choose a
constant $c>0$ such that
\begin{align*}
  \{t\in\R:\, \abs{t}\le c+\epsilon\}\subseteq D(h).
\end{align*}
for a constant $\epsilon>0$. Let $D(a)=\{v\in U:\,\abs{\scapro{v}{a}}\le c\}$. Since for every $u\in B_U$ we have
\begin{align*}
  \abs{\scapro{u}{a_n}-\scapro{u}{a}}\le \norm{a_n-a},
\end{align*}
we can conclude that there exists $n_0\in\N$ such that $u\in D(a)\cap B_U$ implies that $\scapro{u}{a}, \scapro{u}{a_n}\in D(h)$ for every $n\ge n_0$. Consequently, we have for
$f_{a,n}(u):=h(\scapro{u}{a_n})-h(\scapro{u}{a}) -
\big(\scapro{u}{a_n}-\scapro{u}{a}\big)\1_B(u)$ and $n\ge n_0$ the implication:
\begin{align*}
f_{a,n}(u)\neq 0
 \;\Rightarrow\; u \in \big( D(a)\cap B_U^c\big) \cup \big( D^c(a)\cap B_U \big) \cup \big(D^c(a)\cap B_U^c \big).
\end{align*}
As above it can be shown that $f_{a,n}$ is dominated by an integrable function on all
three sets and therefore, Lebesgue's theorem on dominated convergence shows that
$d_{\nu}$ is continuous.

It follows for $h^\prime(s):=s\1_{B_{\R}}(s)$ from \eqref{eq.psinegative} that
\begin{align*}
  f:\R\to\C  \qquad f(s_0,t):=-\widetilde{\psi}_{h^\prime}(s_0,t)
\end{align*}
is negative-definite for each $s_0\in\R$. For  $z_1,\dots ,z_n\in \C$, $a_1,\dots, a_n\in
U^\ast$ we have:
\begin{align*}
&\sum_{i,j=1}^n -z_i\bar{z}_j
 \left(id_{\nu}(a_i-a_j)+\int_U\psi_h(\scapro{u}{a_i-a_j})\,\nu(du)\right)\\
 &\qquad= \sum_{i,j=1}^n -z_i\bar{z}_j \int_U
    \left(e^{i\scapro{u}{a_i-a_j}}-1-i\scapro{u}{a_i-a_j}\1_{B_U}(u)\right)\,\nu(du)\\
  &\qquad= \sum_{i,j=1}^n -z_i\bar{z}_j \int_U
    \left(e^{i\tfrac{\scapro{u}{a_i-a_j}}{\norm{u}}\norm{u}}-1-i\tfrac{\scapro{u}{a_i-a_j}}{\norm{u}}\norm{u}\1_{B_{\R}}(\norm{u})\right)\,\nu(du)\\
&\qquad  = \int_U \sum_{i,j=1}^n z_i\bar{z}_j f\big(\norm{u},
\tfrac{1}{\norm{u}}\left(\scapro{u}{a_i}-\scapro{u}{a_j}\right)\big)\,\nu(du)\le 0.
\end{align*}
Theorem \ref{th.characteristics} implies that there exists an infinitely divisible
cylindrical probability measure $\mu$ with cylindrical characteristics $(d_\nu,0,\nu)_h$. In order to
show that $\mu$ is continuous let $a_n\to a_0$ in $U^\ast$. For a bounded continuous
function $f:\R\to\R$ define
\begin{align*}
  g_n:U\to\R,\qquad
   g_n(u):=f(\scapro{u}{a_n})\,\left(\abs{\scapro{u}{a_n}}^2\wedge 1\right)
\end{align*}
for $n\in \N\cup\{0\}$. It follows that each $g_n\in L_\nu^1(U,\Borel(U))$ and
\begin{align*}
  \abs{g_n(u)}\le \norm{f}_\infty (1+c)\left(\abs{\scapro{u}{a_0}}^2\wedge 1\right)
\end{align*}
for a constant $c>0$. Lebesgue's theorem on dominated convergence implies that
\begin{align*}
  \lim_{n\to\infty} \int_{U} g_n(u)\,\nu(du)=\int_U g_0(u)\,\nu(du),
\end{align*}
which shows that
\begin{align}\label{eq.auxweak2}
   \left(\abs{s}^2\wedge 1\right)\, (\nu\circ a_n^{-1})(ds)
           \to \left(\abs{s}^2\wedge 1\right)\, (\nu\circ a^{-1}_0)(ds)
            \text{ weakly}.
\end{align}
Lemma \ref{le.continuouscylinfdiv}  implies that $\mu$ is continuous and thus $\mu$
is regular continuous by Lemma \ref{le.regular}.
\end{proof}

A cylindrical L{\'e}vy measure which extends to a $\sigma$-finite measure on $\Borel(U)$
is an obvious candidate to be a L{\'e}vy measure in the usual sense. We recall the
definition from Linde \cite{Linde}: a $\sigma$-finite measure $\nu$ on $\Borel(U)$ is
called a {\em L{\'e}vy measure} if
 \begin{enumerate}
   \item[{\rm (a)}] $\displaystyle \int_{U}\big(\scapro{u}{a}^2\wedge 1\big)\,
   \nu(du)<\infty $ for all $a\in U^\ast$;
   \item[{\rm (b)}] there exists a Radon probability  measure $\mu$ on $\Borel(U)$ with characteristic
   function
   \begin{align}\label{eq.Radonchar}
     \phi_\mu:U^\ast\to \C, \quad
      \phi_\mu(a)=\exp\left(\int_U
      \left(e^{i\scapro{u}{a}}-1-i\scapro{u}{a}\1_{B_U}(u)\right)\,\nu(du)\right).
   \end{align}
 \end{enumerate}
In fact, this is rather a result (Theorem 5.4.8) in Linde \cite{Linde} than his
definition. Note furthermore, that this definition includes already the requirement that
a Radon probability measure on $\Borel(U)$ exists with the corresponding characteristic
function. In general, no conditions on a measure $\nu$ are known which guarantee
that $\nu$ is a L{\'e}vy measure. In particular, the condition
\begin{align*}
  \int_{U}\left(\norm{u}^2\wedge 1\right)\,\nu(du)<\infty
\end{align*}
is sufficient and necessary in Hilbert spaces, but in general spaces it is
even neither sufficient nor necessary, such as in the space of continuous
functions on $[0,1]$, see Araujo \cite{Araujo75}.

\begin{corollary}\label{co.Levymeasure2}
  Let $\nu$ be a $\sigma$-finite measure on $\Borel(U)$ and $h$ be a truncation function.
  Then the following are equivalent:
  \begin{enumerate}
    \item[{\rm (a)}] $\nu$ is a L{\'e}vy measure;
    \item[{\rm (b)}] there exists an infinitely divisible cylindrical probability measure $\mu$ with cylindrical characteristics
    $(d_\nu,0,\nu)_h$ which extends to a Radon measure on $\Borel(U)$.
  \end{enumerate}
In this situation, the Radon probability measure with characteristic function \eqref{eq.Radonchar}
corresponding to the L{\'e}vy measure
$\nu$ coincides with the Radon extension of $\mu$.
\end{corollary}
\begin{proof}
It is easily seen that the characteristic function of the cylindrical measure with cylindrical characteristics
$(d_\nu,0,\nu)_h$ is of the form \eqref{eq.Radonchar}. Consequently, (b) implies (a).  If $\nu$ is a L{\'e}vy measure Proposition 5.4.5 in \cite{Linde} guarantees that
  $\nu(B_U^c)<\infty$. Theorem \ref{th.Levymeasure1} implies that  there exists a cylindrical probability
  measure with cylindrical characteristics $(d_\nu,0,\nu)_h$ which extends to a Radon probability measure because
  its characteristic function is of the form \eqref{eq.Radonchar}.
\end{proof}

\begin{remark}
If $\nu$ is a L{\'e}vy measure and $\mu$ the infinitely divisible Radon probability measure with characteristic function
\eqref{eq.Radonchar}  one calls the triplet $(0,0,\nu)$ the characteristics of $\mu$. However, according to Corollary
\ref{co.Levymeasure2} the measure $\mu$  considered as an infinitely divisible cylindrical
probability measure has the cylindrical characteristics $(d_\nu,0,\nu)_h$. Even if we choose the truncation function as $s\mapsto h(s):=s\1_{B_{\R}}(s)$ the entry $d_\nu$ does not vanish. This asymmetry
illustrates the interaction of the components $p$ and $\nu$ of the  cylindrical
characteristics $(p,0,\nu)$ of an infinitely divisible cylindrical measure.
Even if $\nu$ is a L{\'e}vy measure and $p=d_\nu$ then the function
   \begin{align*}
     a\mapsto \kappa(a):=-\left(\int_U
     \left(e^{-i\scapro{u}{a}}-1-i\scapro{u}{a}\1_{B_U}(u)\right)\,\nu(du)\right)
   \end{align*}
 is negative-definite by Bochner's theorem and the Schoenberg's correspondence. But
 although
 \begin{align*}
   \kappa(a)=-\left(id_\nu (a) + \int_U \psi_h(\scapro{u}{a})\,\nu(du)\right)
 \end{align*}
none of the both summands in this representation respectively are negative-definite in
general.
\end{remark}

In general, the condition (b) in Corollary \ref{co.Levymeasure2} might be verified by
applying Prohorov's theorem, \cite[Th.VI.3.2]{Vaketal}, and proving that
the cylindrical measure $\mu$ is tight. In Sazanov spaces this is simplified:
\begin{remark}
  If $U$ is a Sazanov space then condition (b) in Corollary \ref{co.Levymeasure2} can be
  replaced by:
  \begin{enumerate}
    \item[{\rm (b$^\prime$)}]
\begin{enumerate}
\item[(i)] there exits an infinitely divisible cylindrical probability measure $\mu$ with
cylindrical characteristics $(d_\nu,0,\nu)$;
  \item[(ii)]      $a\mapsto \kappa (a)=-\left(id_{\nu}(a)+\int_U \psi_h(\scapro{u}{a})\,\nu(du)\right)$
              is continuous in an admissible topology.
  \end{enumerate}
  \end{enumerate}
\end{remark}

\begin{example}
If $U$ is a Hilbert space then the Sazanov topology is admissible. If $(d_\nu,0,\nu)_h$
is a cylindrical characteristics the function $\kappa$ is necessarily negative-definite
by Theorem \ref{th.characteristics} and if it is also continuous in the Sazanov topology
one obtains the well known L{\'e}vy-Khintchine formula in Hilbert spaces, see
\cite[Th.6.4.10]{Parth67}.
\end{example}

{\it Acknowledgement} The author thanks Dave Applebaum for some very helpful comments and
discussions.

\bibliographystyle{plain}
\bibliography{cylindricalLevy}


\end{document}